\documentclass[11pt]{article}
\usepackage{amsmath}
\usepackage{amssymb}
\usepackage{amsthm}
\usepackage{lineno}
\usepackage{graphicx}
\usepackage{caption}
\usepackage{subcaption}
\usepackage{enumerate}
\usepackage{enumitem}
\usepackage[font=small,skip=0pt]{caption}
\usepackage{xcolor}
\usepackage{verbatim}
\usepackage{geometry}
\usepackage{wrapfig}
\usepackage{color,soul}
\usepackage{float}
\usepackage[stable]{footmisc}
\usepackage{lineno}
\numberwithin{equation}{section}
\newtheorem{Theorem}{Theorem}

\newtheorem*{Note*}{Note}
\newtheorem{Proposition}{Proposition}
\newtheorem{Remark}{Remark}
\newtheorem{Lemma}{Lemma}

\newtheorem*{Recall*}{Recall}
\newcommand{\geqs}{\geqslant}
\newcommand{\leqs}{\leqslant}
\graphicspath{{Figures/}}
\usepackage{color}

\definecolor{ao(english)}{rgb}{0.0, 0.5, 0.0}

\title{Parisian ruin for the dual risk process in discrete-time}
\author{
        Zbigniew Palmowski$^{a,}$\footnote{First author. E-mail address: zbigniew.palmowski@gmail.com}\,, Lewis Ramsden$^{b,}$\footnote{Corresponding author. E-mail address: L.M.Ramsden@liverpool.ac.uk} \, and Apostolos D. Papaioannou$^{b,}$\footnote{Third author. E-mail address: papaion@liverpool.ac.uk}       \\
      \\ $^a$Department of Applied Mathematics\\
       Wroc\l{}aw University of Science and Technology \\
        Wroc\l{}aw, Poland; \\
               $^b$Institute for Financial and Actuarial Mathematics \\
                Department of Mathematical Sciences\\
        University of Liverpool\\
        Liverpool, L69 7ZL, United Kingdom}
\date{}

\begin{document}
\maketitle

\begin{abstract}

In this paper we consider the Parisian ruin probabilities for the dual risk model in a discrete-time setting. By exploiting the strong Markov property of the risk process we derive a recursive expression for the finite-time Parisian ruin probability, in terms of classic discrete-time dual ruin probabilities. Moreover, we obtain an explicit expression for the corresponding infinite-time Parisian ruin probability as a limiting case. In order to obtain more analytic results, we employ a conditioning argument and derive a new expression for the classic infinite-time ruin probability in the dual risk model and hence, an alternative form of the infinite-time Parisian ruin probability. Finally, we explore some interesting special cases, including the Binomial/Geometric model, and obtain a simple expression for the Parisian ruin probability of the Gambler's ruin problem.
\end{abstract}

{\bf Keywords:} Dual risk model, Discrete-time, Ruin probabilities, Parisian ruin, Binomial/Geometric Model, Parisian Gambler's Ruin.

\section{Introduction}

The compound binomial model, first proposed by Gerber (1988), is a discrete-time analogue of the classic Cram\'er-Lundberg risk model which provides a more realistic analysis to the cash flows of an insurance firm. The model has attracted attention since its introduction due to the recursive nature of the results, which are readily programmable in practise, and as a tool to approximate the continuous-time risk model as a limiting case (for details see Dickson (1994)). In the compound binomial risk model, it is assumed that income is received via a periodic premium of size one, whilst the initial reserve and the claim amounts are assumed to be integer valued. That is, the reserve process of an insurer, denoted $\{R_n\}_{n \in \mathbb{N}}$, is given by
\begin{linenomath*}
\begin{equation}
\label{eqCDTRM1}
R_n=u+n-X_n,
\end{equation}
\end{linenomath*}
where $u\in \mathbb{N}$ is the insurers initial reserve and
\begin{linenomath*}\begin{equation}
\label{eqClaims}
X_n=\sum_{i=1}^n Y_i, \quad n=1,2,3, \ldots,
\end{equation}\end{linenomath*}
denotes the aggregate claim amount up to period $n\in \mathbb{N}$, with $X_0=0$. Further, it is assumed that the random non-negative claim amounts, namely $Y_i$, $i=1,2,\ldots$, are independent and identically distributed (i.i.d.) random variables with probability mass function (p.m.f.) $p_k=\mathbb{P}(Y=k)$, for $k=0,1,2, \ldots$, and finite mean $\mathbb{E}(Y_1)<\infty$. We point out, due to its importance in the following, that the claim amounts $Y_i$, $i=1,2,\ldots$ have a mass point at zero with probability $p_0>0$.

Let $T$ denote the time to ruin for the discrete-time risk model given in Eq. \eqref{eqCDTRM1}, defined by
\begin{linenomath*}\begin{equation*}
\label{eqTimeRuin}
T=\inf \{n\in \mathbb{N}: R_n \leqs 0\},
\end{equation*}\end{linenomath*}
where $T=\infty$ if $R_n>0$ for all $n\in \mathbb{N}$. Note that this definition is consistent with Gerber (1988), whilst other authors define the ruin time when the reserve takes strictly negative values (see e.g.\,Willmot (1993)). Then, the finite-time ruin probability, from initial reserve $u \in \mathbb{N}$, is defined by
\begin{linenomath*}\begin{equation*}
\label{CDTRuin}
\psi(u,t)=\mathbb{P}(T< t \, \big| R_0=u), \qquad t\in \mathbb{N},
\end{equation*}\end{linenomath*}
 with corresponding finite-time survival probability $\phi(u,t)=1-\psi(u,t)$. The finite-time ruin probability of the discrete-time risk model was first studied in Willmot (1993), where explicit formulas are derived using generating functions. Later, Lef\`evre and Loisel (2007) derive a seal-type formula based on the ballot theorem (see Tak\'acs (1962)) and a Picard-Lef\`evre-type formula for the corresponding finite-time survival probability, namely $\phi(u,t)$. For further results on finite-time probabilities see Li and Sendova (2013) and references therein. The finite-time ruin probabilities, in general, prove difficult to tackle and the literature on the subject remains few.

On the other hand, the infinite-time ruin probability, defined as the limiting case i.e. $\psi(u)=\lim_{t\rightarrow \infty} \psi(u,t)$, has been considered by several authors e.g. Gerber (1988), Michel (1989), Shiu (1989) and Dickson (1994), among others, where numerous alternative methods have been employed to derive explicit expressions. Further references for related results such as; the discounted probability of ruin, the deficit and surplus prior to ruin and the well known Gerber-Shiu function, to name a few, can be found in Cheng et al.\,(2000), Cossette et al.\,(2003, 2004, 2006), Dickson (1994), Li and Garrido (2002), Pavlova and Willmot (2004), Wu and Li (2009) and Yuen and Guo (2006). For a full comprehensive review of the discrete-time literature refer to Li et al.\,(2009), and references therein.

One limitation of the discrete-time risk model \eqref{eqCDTRM1}, as pointed out by Avanzi et al.\,(2007), is that depending on the line of business there are companies which are subject to a constant flow of expenses and receive income/gains as random events. For instance, pharmaceutical or petroleum companies, where the random gains come from new invention or discoveries, require an alternative to the compound binomial risk model such that the reserve process, namely $\{R^*_n\}_{n\in \mathbb{N}}$, is defined by
\begin{linenomath*}\begin{equation}
\label{eqDTDRP}
R^*_n=u-n+X_n,
\end{equation}\end{linenomath*}
where $\{X_k\}_{k\in\mathbb{N}^+}$ has the same form as Eq.\,\eqref{eqClaims}. This model is known as the \textit{discrete-time dual risk model}. The continuous analogue of the dual risk model has been considered by various authors, with the majority of focus in dividend problems (see Avanzi et al.\,(2007), Bergel et al.\,(2016), Cheung and Drekic (2008), Ng (2009) and references therein). Additionally, Albrecher et al.\,(2008) considered the continuous-time dual risk model under a loss-carry forward tax system, where, in the case of exponentially distributed jump sizes, the infinite-time ruin probability is derived in terms of the ruin probability without taxation.  However, the dual risk problem in discrete-time remains to be studied.

For convenience, throughout the remainder of this paper, we use the notation $\mathbb{P}(\cdot \,| R^*_0=u) = \mathbb{P}_u(\cdot)$ and $\mathbb{P}_0(\cdot)=\mathbb{P}(\cdot)$.

 The finite-time ruin probability, for the dual risk process given in Eq.\,\eqref{eqDTDRP}, is defined in a similar way to the discrete-time risk model defined in Eq.\,\eqref{eqCDTRM1}. That is, the finite-time ruin probability is defined as the probability that the risk reserve process $\{R^*_n\}_{n \in \mathbb{N}}$ attains a non-positive level before some pre-specified time horizon $t\in \mathbb{N}$, from initial capital $u \in \mathbb{N}$. Since the reserve process for the dual risk model, defined in Eq.\,\eqref{eqDTDRP}, experiences deterministic losses of one per period, it follows that the probability of experiencing a non-positive level is equivalent to the probability of hitting the zero level. Thus, let us denote the time to ruin for the dual risk model, given in Eq.\,\eqref{eqDTDRP}, by $\tau^*$, defined by
\begin{linenomath*}\begin{equation*}
\label{eqTimeRuinDual}
\tau^*=\inf \{ n \in \mathbb{N}  : R^*_n=0 \}.
\end{equation*}\end{linenomath*}
 Then, the finite-time dual ruin probability is given by
\begin{linenomath*}\begin{equation}
\label{eqRuinDual}
\psi^*(u,t)=\mathbb{P}(\tau^* < t\, \big| R^*_0=u),
\end{equation}\end{linenomath*}
with the infinite-time dual ruin probability, as above, defined as the limiting case i.e. $\psi^*(u)=\lim_{t\rightarrow \infty} \psi^*(u,t)$. It is clear  that $\tau^* \geqs u$ (due to the deterministic losses of one per period). Finally, it is assumed that the net profit condition holds i.e. $\mu=\mathbb{E}(Y_1)>1$, such that $R^*_n \rightarrow +\infty$ as $n\rightarrow \infty$. This condition ensures that the dual ruin probability is not certain.

The aim of this paper is to extend the notion of ruin to the so-called Parisian ruin, which occurs if the process $\{R^*_n\}_{n\in \mathbb{N}}$ is strictly negative for a fixed number of periods $r \in \{1,2, \ldots \}$ and derive recursive and explicit expressions for the Parisian ruin probability in finite and infinite-time. The idea of Parisian ruin follows from Parisian stock options, where prices are activated or cancelled when underlying assets stay above or below a barrier long enough (see Chesney et al.\,(1997) and Dassios and Wu (2009)). The time of Parisian ruin, in the discrete-time dual risk model, is defined as
\begin{linenomath*}\begin{equation*}
\label{eqParRuinTime}
\tau^r = \inf \{ n \in \mathbb{N} : n - \sup\{ s< n : R^*_s = -1, R^*_{s-1}=0 \}=r \in \mathbb{N}^+, \, R^*_n < 0 \},
\end{equation*}\end{linenomath*}
with finite and infinite-time Parisian ruin probabilities defined by
\begin{linenomath*}\begin{equation*}
\label{eqParRuinDual}
\psi^*_r(u,t)=\mathbb{P}_u(\tau^r < t),
\end{equation*}\end{linenomath*}
and
\begin{linenomath*}\begin{equation*}
\label{eqParRuinDual1}
\psi^*_r(u)=\lim_{t\rightarrow \infty} \psi^*_r(u,t),
\end{equation*}\end{linenomath*}
respectively. We further define the corresponding finite and infinite-time Parisian survival probabilities by $\phi^*_r(u,t)= \mathbb{P}_u(\tau^r \geqs t)=1-\psi^*_r(u,t)$ and $\phi^*_r(u)= 1-\psi^*_r(u)$.

 The extension from classical ruin to Parisian ruin was first proposed, in a continuous time setting, by Dassios and Wu (2011) for the compound Poisson risk process with exponential claim sizes. In this setting they derive expressions for the Laplace transform of the time and probability of Parisian ruin. Further, Czarna and Palmolski (2011) and Loeffen et al.\,(2013) have derived results for the Parisian ruin in the more general case of spectrally negative L\'evy processes. More recently, Czarna et al.\,(2016) adapted the Parisian ruin problem to a discrete-time risk model, as in Eq.\,\eqref{eqCDTRM1}, where finite and infinite-time expressions for the ruin probability are derived, along with the light and heavy-tailed asymptotic behaviour.

 The paper is organised as follows. In Section 2, we exploit the strong Markov property of the risk process to derive a recursive formula for the finite-time Parisian ruin probability, with general initial reserve, in terms of the dual ruin probability defined in Eq.\,\eqref{eqRuinDual} and the Parisian ruin probability with zero initial reserve. For the latter risk quantity, we show this can be calculated recursively. In Section 3, we obtain a similar expression for the corresponding infinite-time Parisian ruin probability, where the Parisian ruin probability with zero initial reserve has an explicit form. In Section 4, we consider an alternative method for calculating the infinite-time dual ruin probability. In Section 5, in order to illustrate the applicability of our recursive type equation, we analyse the Binomial/Geometric model, as a special case. Finally, in Section 6, we derive an explicit expression for the Parisian ruin probability to the well known Gambler's ruin problem.
\section{Finite-time Parisian ruin probability}
\label{SecPARRuin}

In this section, we derive an expression for the finite-time Parisian survival probability, $\phi^*_r(u,t)$, for the dual risk model given in Eq.\,\eqref{eqDTDRP}, for general initial reserve $u \in \mathbb{N}$.

First note that, since the dual risk process, $\{R^*_n\}_{n\in \mathbb{N}}$, experiences only positive random gains and losses occur at a rate of one per period, it follows that $\phi^*_r(u,t)=1$, when $t \leqs u + r +1$. Now, for $t > u + r +1$, by conditioning on the time to ruin, namely $\tau^*$, using the strong Markov property and the fact that $\mathbb{P}_u(\tau^* = k)=0$ for $k< u$, we have

\begin{linenomath*}\begin{equation}
\label{eqFinR1}
\phi^*_r(u,t)=\sum_{k=u}^{t-r-2} \mathbb{P}_u(\tau^* =k)\phi^*_r(0,t-k) + \phi^*(u,t-r-1).
\end{equation}\end{linenomath*}

\noindent Note that the finite-time dual survival probability is given by $\phi^*(u,t)=1-\psi^*(u,t)=1-\sum_{k=0}^{t-1} \mathbb{P}_u(\tau^*=k)$. Thus, from the form of Eq.\,\eqref{eqFinR1}, in order to obtain an expression for the Parisian survival probability, $\phi^*_r(u,t)$, we need only to derive expressions for $ \mathbb{P}_u(\tau^* =k)$ and the Parisian survival probability with zero initial reserve, namely $\phi^*_r(0,t)$.

\begin{Lemma}
\label{Lem2}
In the discrete-time dual risk model, the probability of hitting the zero level from initial capital $u\in \mathbb{N}$, in $n \in \mathbb{N}$ periods, namely $\mathbb{P}_u(\tau^*=n)$, is given by
\begin{linenomath*}\begin{equation}
\label{eqCLRuin}
\mathbb{P}_u(\tau^*=n) =  \frac{u}{n}p_{n-u}^{*n}, \quad n \geqs u,
\end{equation}\end{linenomath*}
where $\{ p_k^{*n} \}_{n \in \mathbb{N}}$ denotes the $n$-th fold convolution of $Y_1$.
\end{Lemma}
\begin{proof}
Consider the discrete-time dual risk process $\{R^*_n\}_{n\in \mathbb{N}}$, defined in Eq.\,\eqref{eqDTDRP}, where
\begin{linenomath*}\begin{equation}
\label{eqSn}
R^*_n=u-S^*_n,
\end{equation}\end{linenomath*}
with $S^*_n=n-X_n$. The `increment' process, $\{S^*_n\}_{n\in \mathbb{N}}$, is equivalent to a discrete-time risk process, given by Eq.\,\eqref{eqCDTRM1}, with initial capital $S^*_0=0$. Therefore, it follows that the dual ruin time, $\tau^*$, is equivalent to the hitting time for the incremental process, $\{S^*_n\}_{n\in \mathbb{N}}$, of the level $u \in \mathbb{N}$ (see Fig:\ref{fig:Example}). Using Proposition 3.1 of Li and Sendova (2013), the result follows.

\begin{figure}
\centering
\begin{subfigure}{.5\textwidth}
  \centering
  \includegraphics[width=7.5cm]{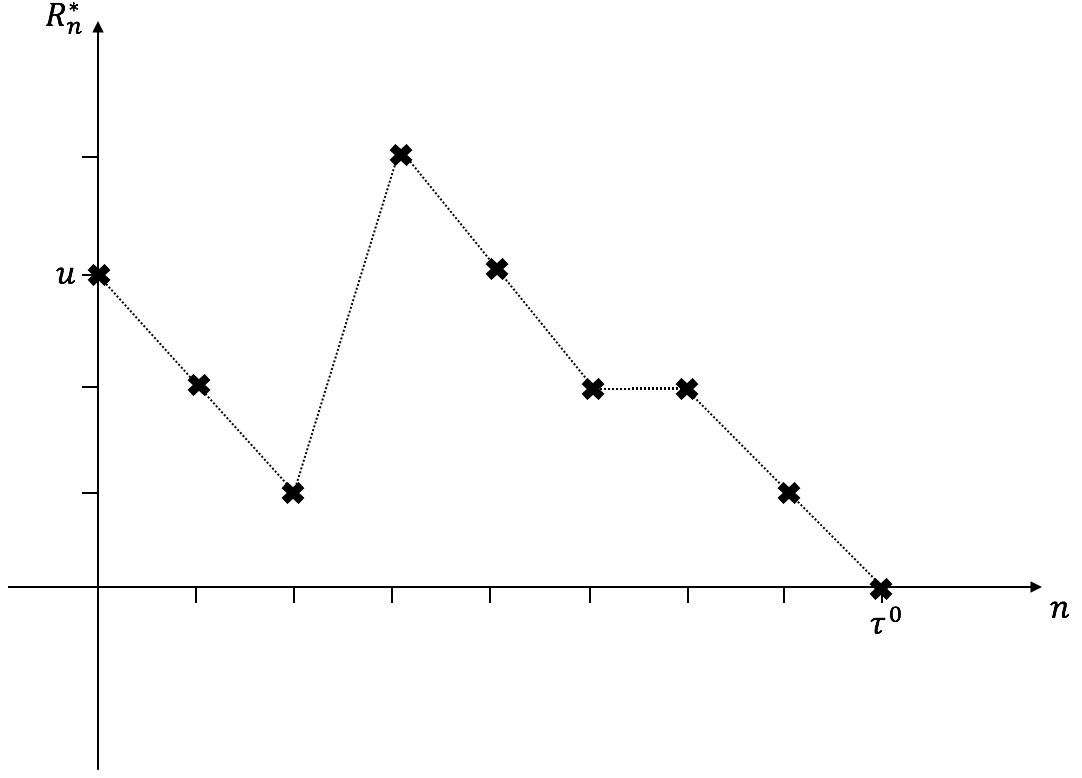}
  \caption{Typical sample path of reserve process\\ $R^*_n$ with initial capital $u\in \mathbb{N}$.}
  \label{fig:sub1}
\end{subfigure}%
\begin{subfigure}{.5\textwidth}
  \centering
  \includegraphics[width=7.5cm]{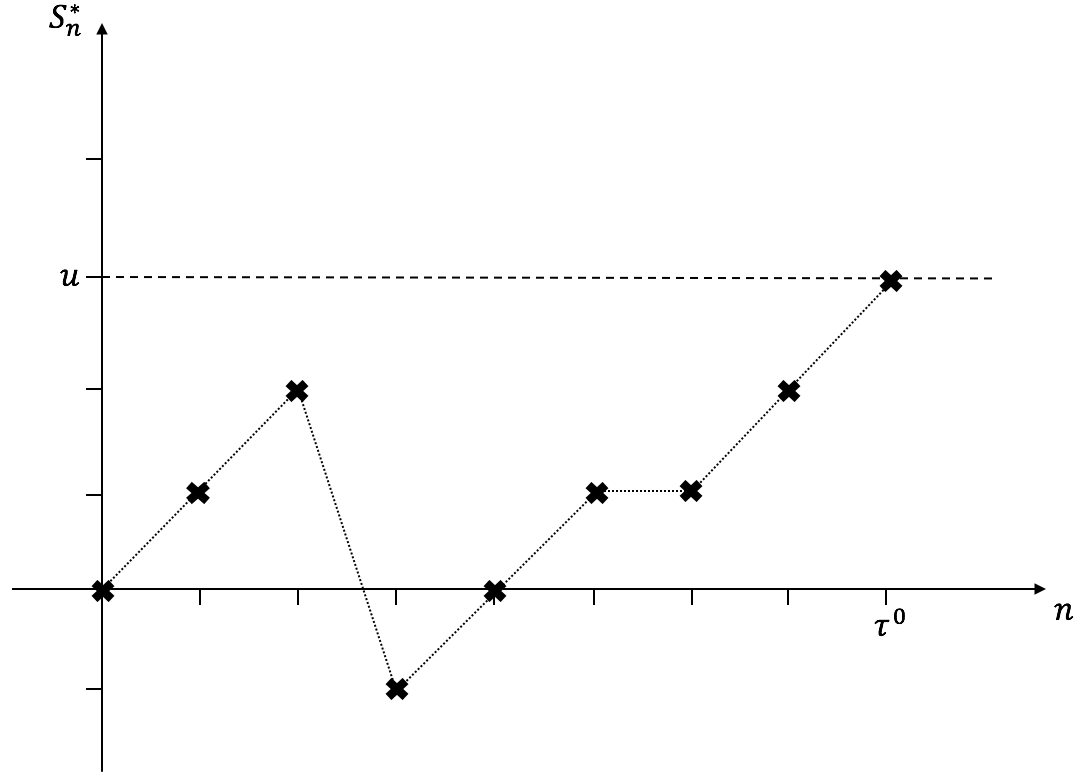}
  \caption{Corresponding sample path of the increment process $S^*_n$ with initial capital $0$.}
  \label{fig:sub2}
\end{subfigure}
\caption{Equivalence between dual risk process and classic risk process.}
\label{fig:Example}
\end{figure}
\end{proof}

Now that we have an expression for $\mathbb{P}_u(\tau^*=k)$, $k\in \mathbb{N}$, and consequently for the finite-time dual survival probability, namely $\phi^*_r(u,t)$, it remains to derive an expression for the finite-time Parisian survival probability for the case where the initial reserve is zero i.e. $R^*_0=0$. Before we begin with deriving an expression for $\phi^*_r(0,t)$, note that in order to avoid Parisian ruin, once the reserve process becomes negative, it will be necessary to return to the zero level (or above) in $r$ time periods or less. Considering this observation, we will introduce another random stopping time, which we name `recovery' time, that measures the number of periods it takes to recover from a deficit to a non-negative reserve. Let us denote the recovery time by $\tau^-$, defined by
\begin{linenomath*}\begin{equation*}
\tau^-=\inf \{ n \in \mathbb{N} : R^*_n \geqs 0, R^*_s < 0 , \forall s < n \}.
\end{equation*}\end{linenomath*}

\noindent Now, consider the dual risk reserve process defined in Eq.\,\eqref{eqDTDRP}, with initial capital $u=0$. If no gain occurs in the first period of time, the risk reserve becomes $R^*_1=-1$ at the end of the period. On the other hand, if there is a random gain of amount $k\in\mathbb{N}^+$ in the first period, the risk reserve becomes $R^*_1=k-1$. Hence, by the law of total probability, we obtain a recursive equation for the finite-time Parisian survival probability, with initial capital zero i.e. $\phi^*_r(0,n)$ (where $n>r+1$), of the form

\begin{linenomath*}\begin{align}
\label{eqFinR2}
\phi^*_r(0,n)&=p_0\, \phi^*_r(-1, n-1)+ \sum_{k=1}^{\infty} p_k \phi^*_r(k-1,n-1) \notag \\
&=p_0 \sum_{s=1}^r \sum_{z=0}^{\infty} \mathbb{P}_{-1}(\tau^- = s, R^*_{\tau^-}=z)\phi^*_r(z,n-s-1) + \sum_{k=0}^{\infty} p_{k+1} \phi^*_r(k,n-1),
\end{align}\end{linenomath*}



\noindent where $\mathbb{P}_{-1}(\tau^- = \cdot, R^*_{\tau^-}=\cdot)$ is the joint density of the recovery time and the size of the overshoot at recovery, given initial capital $u=-1$.



In order to complete the above expression for $\phi^*_r(0,n)$, we need first to derive an expression for $\mathbb{P}_{-1}(\tau^- = \cdot, R^*_{\tau^-}=\cdot)$, which is given in the following Lemma.

\begin{Lemma}
\label{Lem1}
For, $n\in \mathbb{N}^+$ and $k\in \mathbb{N}$, the joint distribution of the recovery time and the overshoot at recovery is given by
\begin{linenomath*}\begin{align}
\label{eqProp1}
\mathbb{P}_{-1}(\tau^- = n, R^*_{\tau^-}=k)= \sum_{j=0}^{n-1}\,  p_j^{*(n-1)}p_{1+n-j+k} -\sum_{j=2}^{n-1} \sum_{i=2}^j \frac{n-j}{n-i} p_{j-i}^{*(n-i)}p_i^{*(i-1)} p_{1+n-j+k},
\end{align}\end{linenomath*}
\end{Lemma}
\begin{proof}
Consider the reflected discrete-time dual risk process, $\{-R^*_n\}_{n \in \mathbb{N}}$, where $\{R^*_n\}_{n \in \mathbb{N}}$ is given in Eq.\,\eqref{eqDTDRP}, with initial capital $u=-1$. Then, it follows that the distribution of the time to cross the time axis and the overshoot of the process at this hitting time are equivalent for both $\{R^*_n\}_{n \in \mathbb{N}}$ and its reflected process $\{-R^*_n\}_{n \in \mathbb{N}}$, which can be described by a discrete-time risk process given in Eq.\,\eqref{eqCDTRM1} (see Fig:\,\ref{fig:Example1}).  Thus, the joint distribution $\mathbb{P}_{-1}(\tau^- = n, R^*_{\tau^-}=k)$ can be found by employing the discrete ruin related quantity from Lemma 2 of Czarna et al.\,(2016). That is, by setting $u=1$ in Eq.\,(4) of Czarna et al.\,(2016), the result follows.
\begin{figure}
\centering
\begin{subfigure}{.5\textwidth}
  \centering
  \includegraphics[width=7.5cm]{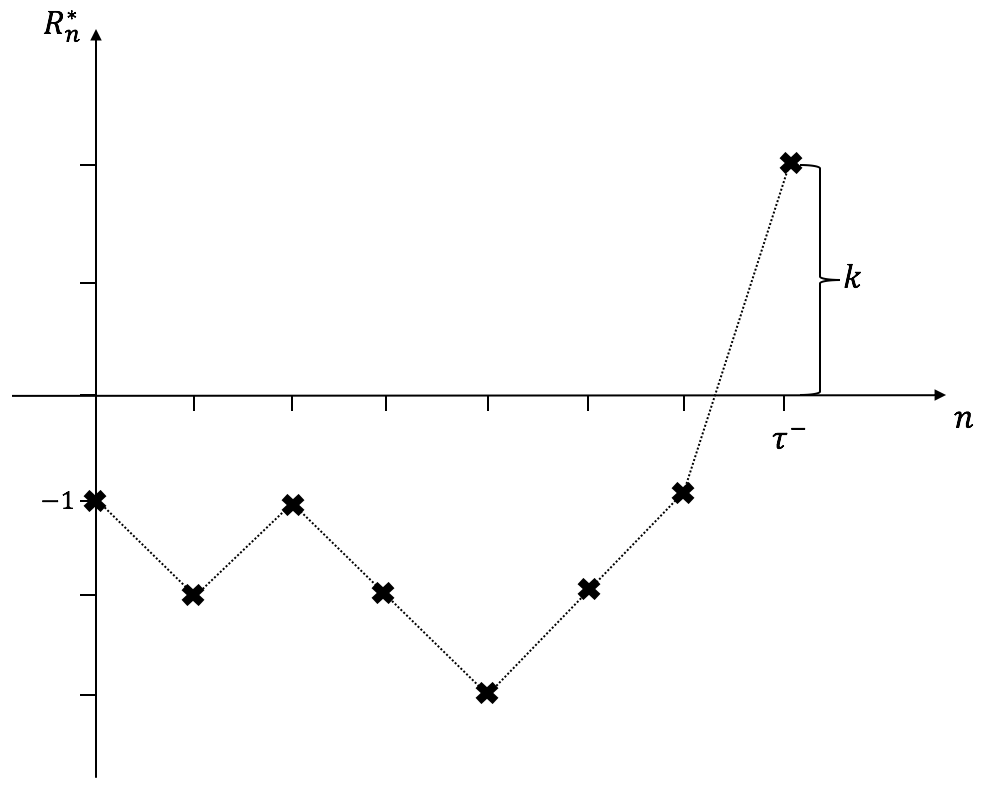}
  \caption{Typical sample path of risk reserve process\\ $R^*_n$ with initial capital $u=-1$.}
  \label{fig:sub1}
\end{subfigure}%
\begin{subfigure}{.5\textwidth}
  \centering
  \includegraphics[width=7.5cm]{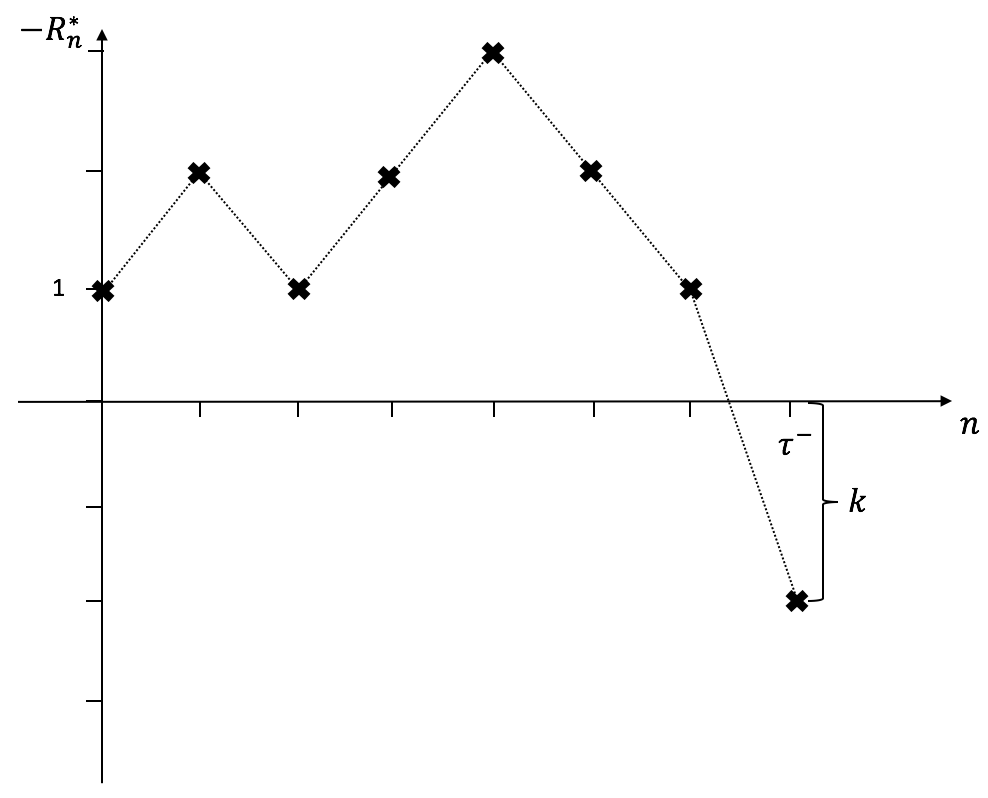}
  \caption{Sample path of the reflected risk reserve process $-R^*_n$ with initial capital $u=1$.}
  \label{fig:sub2}
\end{subfigure}
\caption{Equivalence between original and reflected risk processes.}
\label{fig:Example1}
\end{figure}
\end{proof}

Finally, substituting the form of $\phi^*_r(u,t)$, given in Eq.\,\eqref{eqFinR1}, into Eq.\,\eqref{eqFinR2}, we obtain an expression for $\phi^*_r(0,n)$, of the form

\begin{linenomath*}\begin{align}
\label{eqFinR6}
\phi^*_r(0,n)&=p_0\sum_{s=1}^r \sum_{z=0}^\infty \mathbb{P}_{-1}(\tau^-=s, R^*_{\tau^-}=z) \phi^*(z, n-s-r-2) \notag \\
&\hspace{10mm}+ p_0\sum_{s=1}^r \sum_{z=0}^\infty \sum_{i=z}^{n-s-r-3}\mathbb{P}_{-1}(\tau^-=s, R^*_{\tau^-}=z)\mathbb{P}_z(\tau^*=i) \phi^*_r(0, n-s-i-1) \notag \\
&\hspace{20mm}+\sum_{k=0}^\infty p_{k+1} \phi^*(k, n-r-2) + \sum_{k=0}^\infty \sum_{i=k}^{n-r-3} p_{k+1} \mathbb{P}_k(\tau^*=i) \phi^*_r(0, n-i-1).
\end{align}\end{linenomath*}

\begin{Remark}
\rm An explicit expression for $\phi^*_r(0,n)$, based on Eq.\,\eqref{eqFinR6}, proves difficult to obtain. However, due to the form of Eq.\,\eqref{eqFinR6}, a recursive calculation for $\phi^*_r(0,n)$ is given by the following algorithm:  \\

\noindent \textbf{Step 1.} For $n=r+2$, in Eq.\,\eqref{eqFinR6}, and using the fact that $\phi^*(u,t)=1$ for $t\leqs u$, we have that
\begin{linenomath*}\begin{align*}
\phi^*_r(0,r+2)&=p_0\sum_{s=1}^r \sum_{z=0}^\infty \mathbb{P}_{-1}(\tau^-=s, R^*_{\tau^-}=z)+1-p_0\\
&=1-p_0\left(1- \sum_{z=0}^\infty \mathbb{P}_{-1}(\tau^- \leqs r, R^*_{\tau^-}=z)\right)\\
&=1-p_0\phi(1,r+1),
\end{align*}\end{linenomath*}
where $\phi(u,t)$ is the classic finite-time survival probability in the compound binomial risk model, which has been extensively studied in the literature, [see Li and Sendova (2013) and references therein] and alternatively can be evaluated using Lemma \ref{Lem1}. \\

\noindent \textbf{Step 2.} Based on the result of step 1, we can compute the following term, i.e. for $n=r+3$, we have
\begin{linenomath*}\begin{align*}
\phi^*_r(0,r+3)&=p_0\sum_{s=1}^r \sum_{z=0}^\infty \mathbb{P}_{-1}(\tau^-=s, R_{\tau^-}=z)+\sum_{k=1}^\infty p_{k+1}+p_1\phi^*_r(0,r+2)\\
&=1-(1+p_1)p_0\phi(1,r+1).
\end{align*}\end{linenomath*}

\noindent \textbf{Step 3.} For $n=r+4$, we have
\begin{linenomath*}\begin{align*}
\phi^*_r(0,r+4)&=p_0\left(\sum_{z=1}^\infty \mathbb{P}_{-1}(\tau^-=1, R_{\tau^-}=z)+\sum_{s=2}^r \sum_{z=0}^\infty \mathbb{P}_{-1}(\tau^-=s, R_{\tau^-}=z)\right) \\
&\hspace{10mm}+p_0\mathbb{P}_{-1}(\tau^-=1, R_{\tau^-}=0)\phi^*_r(0,r+2)+p_2\phi^*(1,2)+\sum_{k=2}^\infty p_{k+1} \\
&\hspace{20mm}+p_1\phi^*_r(0,r+3)+p_2\mathbb{P}_1(\tau^*=1)\phi^*_r(0,r+2)\\
&=p_0\left(\psi(1,r+1)-\mathbb{P}_{-1}(\tau^-=1, R_{\tau^-}=0)\right) \\
&\hspace{10mm}+p_0\mathbb{P}_{-1}(\tau^-=1, R_{\tau^-}=0)\phi^*_r(0,r+2)+p_2\left(1-p_0\right)+1-(p_0+p_1+p_2) \\
&\hspace{20mm}+p_1\phi^*_r(0,r+3)+p_2p_0\phi^*_r(0,r+2).
\end{align*}
Employing the results of steps 1 and 2 and using the fact that $\mathbb{P}_{-1}(\tau^-=1, R_{\tau^-}=0)=p_2$, by Lemma \ref{Lem1}, after some algebraic manipulations we obtain
\begin{align*}
\phi^*_r(0,r+4)&=1-\left[1+2p_0p_2+p_1(1+p_1)\right]p_0\phi(1,r+1).
\end{align*}\end{linenomath*}

\noindent Thus, based on the above steps, it can be seen that $\phi_r^*(0,r+k)$, for $k=2,3,\ldots$, can be evaluated recursively for each value of $k$ in terms of the mass functions, $p_k$, and the classic ruin quantity $\phi(1,r+1)$.
\end{Remark}

\begin{Theorem} For $u\in \mathbb{N}$, the finite-time Parisian ruin probability $\psi^*_r(u,t)=0$ for $ t\leqs u+r+1$ and for $t>u+r+1$, is given by
\begin{linenomath*}\begin{equation}
\label{eqThm1b}
\psi^*_r(u,t)=\sum_{k=u}^{t-r-2} \mathbb{P}_u(\tau^* =k)\psi^*_r(0,t-k),
\end{equation}\end{linenomath*}
where $\mathbb{P}_u(\tau^* =k)$ is given in Lemma \ref{Lem2} and the initial value $\psi^*_r(0,n)$ can be found recursively from Eq.\,\eqref{eqFinR6}.
\end{Theorem}

In the next subsection, we use the above expressions to derive results for the infinite-time Parisian ruin probabilities, for which, as will be seen, a more analytic expression can be found.

\section{Infinite-time Parisian ruin probability}

In this section we derive an explicit expression for the infinite-time Parisian survival (ruin) probabilities using the arguments of the previous section. First, let us recall that the infinite-time Parisian survival probability is defined as $\phi_r^*(u)=\lim_{t\rightarrow \infty} \phi_r^*(u,t)$, with the infinite-time dual ruin quantities being defined in a similar way i.e. $\phi^*(u)=\lim_{t\rightarrow \infty} \phi^*(u,t)$. Then, it follows by taking the limit $t\rightarrow \infty$, with $t \in \mathbb{N}$, Eq.\,\eqref{eqFinR1} reduces to
\begin{linenomath*}\begin{equation}
\label{eqParEQ1}
\phi^*_r(u) = \psi^*(u)\phi^*_r(0)+ \phi^*(u),
\end{equation}\end{linenomath*}

\noindent where $\phi^*_r(0)$ is the infinite-time probability of Parisian survival with zero initial reserve and satisfies $\phi^*_r(0)=\lim_{t \rightarrow \infty} \phi^*_r(0,t)$, where $\phi^*_r(0,t)$ is given by Eq.\,\eqref{eqFinR2}. Thus, $\phi^*_r(0)$ is given by

\noindent
\begin{linenomath*}\begin{equation*}
\label{eqParEQ2}
\phi^*_r(0) = p_0 \sum_{z=0}^\infty \mathbb{P}_{-1} (\tau^- \leqs  r, R^*_{\tau^-}=z) \phi^*_r(z) +  \sum_{j=0}^\infty p_{j+1} \phi^*_r(j),
\end{equation*}\end{linenomath*}
or equivalently

\begin{linenomath*}\begin{equation}
\label{eqParEQ3}
\phi^*_r(0) =  \sum_{k=0}^\infty \left( p_0 \mathbb{P}_{-1} (\tau^- \leqs  r, R^*_{\tau^-}=k) + p_{k+1} \right) \phi^*_r(k),
\end{equation}\end{linenomath*}
where $\mathbb{P}_{-1} (\tau^- \leqs  r, R^*_{\tau^-}=k)$ can be obtained from the result of Lemma \ref{Lem1}, i.e.
$$\mathbb{P}_{-1} (\tau^- \leqs  r, R^*_{\tau^-}=k)=\sum_{s=1}^r \mathbb{P}_{-1} (\tau^- =  s, R^*_{\tau^-}=k).$$

\noindent Considering the first term of the summation in the right hand side of Eq.\,\eqref{eqParEQ3} and solving with respect to $\phi^*_r(0)$, we get an explicit representation for $\phi^*_r(0)$, given by
\begin{linenomath*}\begin{equation}
\label{eqParEQ4}
\phi^*_r(0) =C^{-1}\sum_{k=1}^\infty \left( p_0 \mathbb{P}_{-1} (\tau^- \leqs  r, R^*_{\tau^-}=k) + p_{k+1} \right) \phi^*_r(k),
\end{equation}\end{linenomath*}
where
\begin{equation*}
\label{eqConst1}
C=1-p_0 \mathbb{P}_{-1}(\tau^- \leqs r, R^*_{\tau^-}=0)-p_1.
\end{equation*}

\noindent Now, since from Lemma \ref{Lem1} we can obtain an expression for the joint distribution of the time of recovery and the overshoot, namely $\mathbb{P}_{-1} (\tau^- \leqs  r, R^*_{\tau^-}=k)$, we can re-write Eq.\,\eqref{eqParEQ4} as

\begin{linenomath*}\begin{equation*}
\label{eqParEQ5}
\phi^*_r(0) =C^{-1}\sum_{k=1}^\infty a_k \phi^*_r(k),
\end{equation*}\end{linenomath*}
where $a_k =  \left( p_0 \mathbb{P}_{-1} (\tau^- \leqs  r, R^*_{\tau^-}=k) + p_{k+1} \right)$. Then, by substituting the general form of the infinite-time Parisian survival probability, given by \eqref{eqParEQ1}, into the above equation, and solving the resulting equation with respect to $\phi^*_r(0)$, we obtain
\begin{linenomath*}\begin{equation}
\label{eqParEQ6}
\phi^*_r(0) =\frac{C^{-1}\sum_{k=1}^\infty a_k \phi^*(k)}{1-C^{-1}\sum_{k=1}^\infty a_k \psi^*(k)}.
\end{equation}\end{linenomath*}

\noindent Note that, unlike for the finite-time case, in the infinite-time case we obtain an explicit expression for the Parisian survival probability, with zero initial reserve, which is given in terms of the infinite-time dual ruin probabilities. Thus, employing Eq.\,\eqref{eqParEQ1} and the result from Lemma \ref{Lem2} we obtain an explicit expression for the infinite-time Parisian survival probability, with general initial reserve $u\in \mathbb{N}$, given in the following Theorem.

\begin{Theorem} For $u\in \mathbb{N}$, the infinite-time Parisian ruin probability $\psi^*_r(u)$, is given by
\label{Thm1}
\begin{linenomath*}\begin{equation}
\label{eqThm1}
\psi^*_r(u)=\psi^*(u)\left(1-\frac{C^{-1}\sum_{k=1}^\infty a_k \phi^*(k)}{1-C^{-1}\sum_{k=1}^\infty a_k \psi^*(k)}\right),
\end{equation}\end{linenomath*}
where
\begin{equation}
\label{eqak1}
a_k= \left( p_0 \mathbb{P}_{-1} (\tau^- \leqs  r, R^*_{\tau^-}=k) + p_{k+1} \right),
\end{equation}
 and
$$C^{-1}=\left(1-p_0 \mathbb{P}_{-1}(\tau^- \leqs r, R^*_{\tau^-}=0)-p_1\right)^{-1}.$$
\end{Theorem}

\begin{proof}
The result follows by combining Eqs.\,\eqref{eqParEQ1} and \eqref{eqParEQ6}, and recalling that $\phi^*_r(u)=1-\psi^*_r(u)$.
\end{proof}

\section{An alternative approach to the infinite-time dual ruin probability}

In this section we analyse an alternative approach, in order to find an explicit expression, for the infinite-time dual ruin probability. The method is based on the fact that the ruin probability $\psi^*(u)$ satisfies a difference equation, where a particular form of the solution is adopted. In the following, we show that this solution is indeed an analytical solution for $\psi^*(u)$ and is unique.

Although the result of Lemma \ref{Lem2} provides us with a general form for $\psi^*(u)=\sum_{k=0}^\infty \mathbb{P}_u (\tau^*=k)$, in terms of convolutions of the p.m.f. of $Y_1$, let us now consider the dual risk reserve process given in Eq.\,\eqref{eqDTDRP} with initial reserve $u+1$, $u\in \mathbb{N}$ and condition on the possible events in the first time period. Then, by law of total probability, we obtain a recursive equation for the infinite-time dual ruin probability, namely $\psi^*(\cdot)$, given by

\begin{linenomath*}\begin{align}
\label{eqDifEq1}
\psi^*(u+1) &= p_0 \psi^*(u)+ \sum_{j=1}^\infty p_j \psi^*(u+j) \\
&= \sum_{j=0}^\infty p_j \psi^*(u+j),
\end{align}\end{linenomath*}
with boundary conditions $\psi^*(0)=1$ and $\lim_{u\rightarrow \infty} \psi^*(u)=0$.

Equation \eqref{eqDifEq1} is in the form of an infinite-order difference (recursive) type equation. Thus, by adopting the general methodology for solving difference equations, we search for a solution of the form
\begin{linenomath*}\begin{equation*}
\psi^*(u)=cA^u,
\end{equation*}\end{linenomath*}
where $c$ and $A$ are constants to be determined. Using the given boundary conditions for $\psi^*(\cdot)$, it follows that the constant $c=1$ and $0\leqs A<1$. That is, the general solution to the recursive Eq.\,\eqref{eqDifEq1} is of the form
\begin{linenomath*}\begin{equation}
\label{eqGenSol}
\psi^*(u)=A^u,
\end{equation}\end{linenomath*}
for some $0\leqs A<1$. Substituting the general solution, given in Eq.\,\eqref{eqGenSol}, into Eq.\,\eqref{eqDifEq1}, yields
\begin{linenomath*}\begin{equation*}
\label{eqDifEq10}
A^{u+1}= \sum_{j=0}^\infty p_j A^{u+j}, \qquad u=0,1,2, \ldots,
\end{equation*}\end{linenomath*}
from which, dividing through by $A^u$ and defining the probability generating function (p.g.f.) of $Y_1$ by $\tilde{p}(z)=\sum_{i=0}^\infty p_iz^i$, we obtain
\begin{linenomath*}\begin{equation}
\label{eqGenSol1}
A= \tilde{p}(A), \qquad 0\leqs A < 1.
\end{equation}\end{linenomath*}

\noindent That is, $0\leqs A<1$ is a solution (if it exists) to the discrete-time dual analogue of Lundberg's fundamental equation, given by
\begin{linenomath*}\begin{equation}
\label{eqDifEq11}
\gamma(z)=0,
\end{equation}\end{linenomath*}
where $\gamma(z):=\tilde{p}(z)-z$.

\begin{Proposition}
\label{Prop1}
In the interval $[0,1)$ there exists a unique solution to the equation $\tilde{p}(z)-z=0$.
\end{Proposition}

\begin{proof}
It follows from the properties of a p.g.f. that
\begin{linenomath*}\begin{align*}
\gamma(0)&=p_0 \geqs 0, \\
\gamma'(0)&=p_1 -1 \leqs 0, \\
\gamma(1)&=0, \\
\gamma'(1)&=\mathbb{E}(Y_1) -1> 0, \\
\gamma''(z)&>0, \quad \forall z\in[0,1).
\end{align*}\end{linenomath*}

\noindent From the above conditions, which show the characterisitics of the function $\gamma(z):=\tilde{p}(z)-z$ (see Fig:\,\ref{fig:Lund}), it follows that there exists a solution to $\gamma(z)=0$ at $z=1$ and a second solution $z=A$, which is unique in the interval $[0,1)$.
\begin{figure}[h!]
\centering
\includegraphics[width=7.5cm]{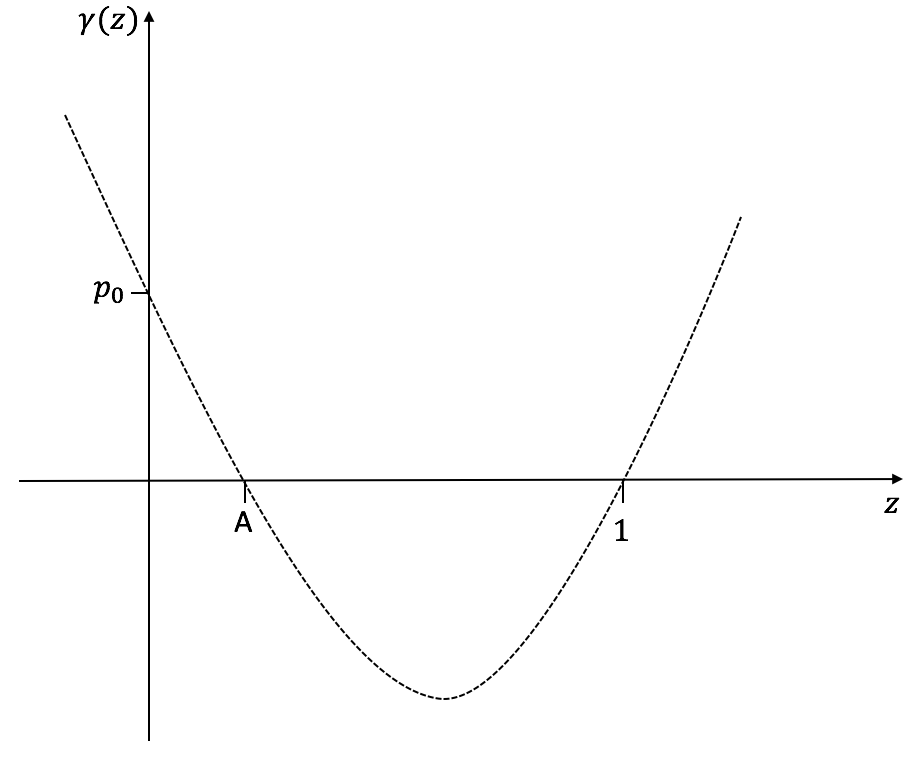}
\caption{Graph of the function $\tilde{p}(z)-z$.}
\label{fig:Lund}
\end{figure}
\end{proof}

\noindent Hence, from Eqs.\,\eqref{eqGenSol}, \eqref{eqGenSol1} and Proposition \ref{Prop1}, we obtain an expression for the infinite-time dual ruin probability, given by the following theorem.
\begin{Theorem}
\label{Thm2}
The infinite-time dual probability of ruin, namely $\psi^*(u)$ for $u \in \mathbb{N}$, is given by
\begin{linenomath*}\begin{equation}
\label{eqCLRuinIT}
\psi^*(u) = A^u,
\end{equation}\end{linenomath*}
where $A$ is the unique solution in the interval $[0,1)$ to the equation $\tilde{p}(z)-z=0$, with $\tilde{p}(z)$ the p.g.f.\,of $Y_1$.
\end{Theorem}

\begin{Remark} \rm We note that the p.g.f. $\tilde{p}(z)$ converges for all $|z| \leqs 1$ and thus, in the interval $z\in[0,1]$ the p.g.f.\,exists (finite) for all probability distributions i.e.\,light and heavy-tailed. Therefore, it follows that Theorem \ref{Thm2} holds for both light and heavy-tailed gain size distributions. 
\end{Remark}
\begin{Remark} \rm
Note that for the general claim size distribution of $Y_i$, one cannot expect the heavy-tailed asymptotics of the Parisian ruin probability 
$\psi_r^*(u)$. Indeed, recall that from \eqref{eqThm1} we have that $\psi_r^*(u)= D \psi^*(u)\leq \psi^*(u)$
for the constant $D=1-\frac{C^{-1}\sum_{k=1}^\infty a_k \phi^*(k)}{1-C^{-1}\sum_{k=1}^\infty a_k \psi^*(k)}$.
Now, observing our discrete process $R^*_n$ at the moments of claim arrivals, we can conclude that:  
\[\psi^*(u)=\mathbb{P}\left(\max_{n\geq 0}\sum_{i=1}^n (T_i-\tilde{Y}_i)>u\right),\]
where $\{T_i\}_{\{i=1,2\ldots\}}$ is a sequence of i.i.d. interarrival times independent of the renormalized sequence of i.i.d. 
claim sizes $\{\tilde{Y}_i\}_{\{i=1,2\ldots\}}$ with the law $\mathbb{P}(\tilde{Y_i}=k)=p_k/(1-p_0)$ for $k=1,2,\ldots$.
In our model the generic interarrvial time $T_i$ has the geometric distribution with the parameter $p_0$ and hence it is 
light-tailed. From general theory of level crossing probabilities by random walks, see e.g. Theorem XIII.5.3 and Remark XIII.5.4 of Asmussen (2003), 
it follows that asymototic tail of the ruin probability $\psi^*(u)$ always decay exponentially fast.
The same concerns then the Parisian ruin probability $\psi_r^*(u)$.
\end{Remark}

\section{Binomial/Geometric model}

In this section, we consider the Binomial/Geometric model as studied by Gerber (1988), Shiu (1989) and Dickson (1998), among others and we derive an exact expression for the infinite-time dual probability of ruin, namely $\psi^*(u)$. Consequently, from Theorem \ref{Thm2} we obtain an expression for the corresponding infinite-time Parisian ruin probability, $\psi^*_r(u)$.

 In the Binomial/Geometric model, it is assumed that the gain size random variables $\{Y_i\}_{i \in \mathbb{N}^+}$ have the form $Y_i=I_i\cdot X_i$, where $I_i$ for $i\in \mathbb{N}^+$, are i.i.d.\,random variables following a Bernoulli distribution with parameter $b\in [0,1]$ i.e. $\mathbb{P}(I_1=1)=1-\mathbb{P}(I_1=0)=b$ and the random gain amount $\{X_k\}_{k\in\mathbb{N}^+}$ are i.i.d.\,random variables following a geometric distribution with parameter $(1-q)\in [0,1]$ i.e. $\mathbb{P}(Y_1=0)=p_0=1-b$ and $\mathbb{P}(Y_1=k)=p_k=bq^{k-1}(1-q)$ for $k\in \mathbb{N}^+$.

\begin{Lemma}
\label{Lem3}
For $u\in \mathbb{N}$, the infinite-time dual ruin probability, $\psi^*(u)$, in the Binomial/Geometric model, with parameters $b\in [0,1]$ and $(1-q)\in[0,1]$ such that $b+q>1$, is given by
\begin{linenomath*}\begin{equation}
\label{eqCLRuinGeo}
\psi^*(u)=\left(\frac{1-b}{q}\right)^u.
\end{equation}\end{linenomath*}
\end{Lemma}

\begin{proof}
From Theorem \ref{Thm2}, the infinite-time dual ruin probability, $\psi^*(u)$, has the form $\psi^*(u)=A^u$, where $0\leqs A<1$, is the solution to $\gamma(z):=\tilde{p}(z)-z=0$, with
 \begin{linenomath*}\begin{equation}
 \label{eqBinGeo1}
 \tilde{p}(z)=1-b+b\tilde{q}(z),
 \end{equation}\end{linenomath*}
 and $\tilde{q}(z)$ is the p.g.f.\,of a geometric random variable, which takes the form
 \begin{linenomath*}\begin{equation}
 \label{eqBinGeo2}
 \tilde{q}(z)=\frac{(1-q)z}{1-qz}.
 \end{equation}\end{linenomath*}
Combining Eqs.\,\eqref{eqBinGeo1} and \eqref{eqBinGeo2} and after some algebraic manipulations, Lundberg's fundamental equation $\gamma(z)=0$, yields a quadratic equation of the form
\begin{linenomath*}\begin{equation*}
 z^2+k_1z+k_2=0,
 \end{equation*}\end{linenomath*}
 where
 \begin{linenomath*}\begin{align*}
 k_1&=\frac{b-1}{q}-1, \\
 k_2&=\frac{1-b}{q}.
 \end{align*}\end{linenomath*}
The above quadratic equation has two roots $z_1=(1-b)/q$ and $z_2=1$. Finally, from the positive drift assumption in the the model set up, we have that $\mathbb{E}(Y_1)=b/(1-q)>1$, from which it follows that $b+q>1$ and the solution $z_1\in [0,1)$. Thus, we have $A=z_1$, since this solution is unique in the interval $[0,1)$ (see Proposition \ref{Prop1}).
\end{proof}

\section{Parisian ruin for the Gambler's ruin problem}
In this section we derive an explicit expression for the infinite-time Parisian ruin probability for one of the more fundamental ruin problems, namely the Gambler's ruin problem. In this model a player makes a bet on the outcome of a random game, with a chance to double their bet with probability $b\in [0,1]$. Ruin in this model is defined as being the event that the player runs out of money at some point (see Feller (1968)).

Mathematically, the Gambler's ruin model can be described by the discrete-time dual risk model, considered in the previous sections, with a loss probability $p_0=1-b$, corresponding win probability $p_2=b$ and $p_k=0$ otherwise. Further, in order to satisfy the net profit condition, and consequently avoid definite ruin over an infinite-time horizon, it follows that $b > 1/2$.

Under these assumptions Lundberg's fundamental equation, $\gamma(z)=0$, produces a quadratic equation of the form
\begin{equation*}
z^2 - \frac{1}{b}z+\frac{1-b}{b}=0,
\end{equation*}
which has solutions $z_1=1$ and $z_2=\frac{1-b}{b}$. From the net profit condition, i.e. $b>1/2$, it follows that $z_2=\frac{1-b}{b} < 1$. Thus, from Theorem \ref{Thm2}, we have that $A=\frac{1-b}{b}$ and the classic Gambler's ruin probability is given by
\begin{linenomath*}\begin{equation}
\label{eqGambler'sRuin}
\psi^*(u)=\left(\frac{1-b}{b}\right)^u,
\end{equation}\end{linenomath*}
as seen in Feller (1968). Finally, from Theorem \ref{Thm1}, the infinite-time Parisian ruin probability for the Gambler's ruin problem is given by the following Proposition.

\begin{Proposition}
\label{PropGR}
The infinite-time Parisian ruin probability to the Gambler's ruin problem, with loss probability $b < 1/2$, is given by
\begin{equation}
\label{eqGR10}
\psi_r^*(u)=\frac{1-bC_1}{1-(1-b)C_1}\left(\frac{1-b}{b}\right)^{u+1} ,
\end{equation}
where
\begin{equation}
\label{eqC1}
C_1=\sum_{n=1}^{r}\,  p_{n-1}^{*(n-1)} -\sum_{n=1}^{r} \sum_{i=2}^{n-1} \frac{1}{n-i} p_{n-1-i}^{*(n-i)}\,p_i^{*(i-1)}.
\end{equation}
\end{Proposition}
\begin{proof}
Using the result of Theorem \ref{Thm1}, and the form of the classic Gambler's ruin problem given by Eq.\,\eqref{eqGambler'sRuin}, it remains to find explicit expressions for the coefficients $a_k$, $k=1,\ldots, \infty$ and the constant $C^{-1}$.

Let us first consider the coefficients $a_k$, given by Eq.\,\eqref{eqak1}, of the form
\begin{equation*}
\label{eqak}
a_k=\left( p_0 \mathbb{P}_{-1} (\tau^- \leqs  r, R^*_{\tau^-}=k) + p_{k+1} \right).
\end{equation*}
Recalling that in the Gambler's ruin problem the p.m.f's of the positive gain sizes i.e. $p_k=0$ for $k \neq 0,2$, it follows that only positive jumps of size $Y_i=2$, for $i\in \mathbb{N}^+$, can occur (with probability $b$) and thus, the joint distribution of recovery and the overshoot at the time of recovery, namely $\mathbb{P}_{-1} (\tau^- \leqs  r, R^*_{\tau^-}=k)=0$, for all $k\neq 0$. Thus, we have that, for $k=1,\ldots, \infty$, $a_k=p_{k+1}$ and it follows
\begin{equation*}
a_k=\begin{cases}
b, \qquad &k=1, \\
0 \qquad &otherwise.
\end{cases}
\end{equation*}
Substituting this into the result of Theorem \ref{Thm1} and after some algebraic manipulations, we obtain
\begin{equation*}
\psi_r^*(u)=\frac{C-b}{C-(1-b)}\left(\frac{1-b}{b} \right)^u,
\end{equation*}
where $C=1-(1-b) \mathbb{P}_{-1}(\tau^- \leqs r, R^*_{\tau^-}=0)$.

Finally, by setting $z=0$ in Eq.\,\eqref{eqProp1} and noticing that, since $p_k=0$, for $k=3,4,\ldots$, only the term $j=n-1$ remains in both summation terms, we obtain
\begin{equation*}
\mathbb{P}_{-1}(\tau^- \leqs r, R^*_{\tau^-}=0)=b \left(\sum_{n=1}^{r}\,  p_{n-1}^{*(n-1)} -\sum_{n=1}^{r} \sum_{i=2}^{n-1} \frac{1}{n-i} p_{n-1-i}^{*(n-i)}\,p_i^{*(i-1)}\right),
\end{equation*}
and it follows that $C=1-b(1-b)C_1$, where $C_1$ is given by Eq.\,\eqref{eqC1}. Finally, the result follows after some algebraic manipulations.
\end{proof}

\section*{Acknowledgments}
\noindent 
The work of Z. Palmowski is partially supported by the National Science Centre under the grant
2016/23/B/HS4/00566 (2017-2020).



\end{document}